\documentclass[preprint,12pt]{elsarticle}

\usepackage{amssymb}
\usepackage{graphicx}
\usepackage{amsmath}
\usepackage{lineno,hyperref}
\usepackage{amsmath}
\usepackage{appendix}
\usepackage{graphicx}
\numberwithin{equation}{section}
\newtheorem{theorem}{Theorem}[section]
\newtheorem{lemma}{Lemma}[section]
\newtheorem{proposition}{Proposition}[section]
\newtheorem{remark}{Remark}[section]

\journal{}

\begin{document}

\begin{frontmatter}



\title{A Free Boundary Problem with a Stefan Condition for a Ratio-dependent Predator-prey Model\tnoteref{mytitlenote}}
\tnotetext[mytitlenote]{This work was supported by \href{http://www.ctan.org/tex-archive/macros/latex/contrib/elsarticle}{NSFC 11671278}.}
\author{Lingyu Liu}
\ead{liu\_lingyu@foxmail.com}
\address{Department of Mathematics, Sichuan University, Chengdu 610065, PR China}

\begin{abstract}
In this paper we study a ratio-dependent predator-prey model with a free boundary causing by both prey and predator over a one dimensional habitat. We study the long time behaviors of the two species and prove a spreading-vanishing dichotomy, namely, as $t$ goes to infinity, both prey and predator successfully spread to the whole space and survive in the new environment, or they spread within a bounded area and die out eventually. Then the criteria governing spreading and vanishing
are obtained. Finally, when spreading occurs, we provide some estimates to the asymptotic spreading speed of $h(t)$.
\end{abstract}

\begin{keyword}
\texttt{free boundary\sep ratio-dependent model\sep spreading-vanishing dichotomy\sep criteria \sep asymptotic speed}
\end{keyword}

\end{frontmatter}


\section{Introduction}
In this paper, we consider the following ratio-dependent predator-prey model,
\begin{equation}\label{Q}
\left\{
\begin{array}{ll}
u_t-u_{xx}=\lambda u-u^2-\frac{buv}{u+mv},    &t>0 ,~0< x<h(t),\\
v_t -dv_{xx}=\nu v-v^2+\frac{cuv}{u+mv},       &t>0 ,~0<x<h(t),\\
u_x=v_x=0,                          &t\geq0,~x=0,\\
u=v=0,~h^\prime (t)=-\mu( u_x+\rho v_x),     &t\geq0,~x= h(t),\\
u(0,x)=u_0(x),~ v(0,x)=v_0(x),        &0\leq x\leq h_0,\\
h(0)=h_0,
\end{array}
\right.
\end{equation}
where $\lambda$, $b$, $m$, $d$, $\nu$, $c$, $\mu$, $\rho$, $h_0$ are given positive constants. $u$ and $v$ stand for prey and predator density, respectively. $x=h(t)$ is the moving boundary determined by $u(t,x)$ and $v(t,x)$, which is the free boundary to be solved. The initial functions $u_0(x)$ and $v_0(x)$ satisfy
$$
u_0, v_0\in \mathcal{C}^2([0,h_0]),~~~u_0(x), v_0(x)>0 , x\in[0,h_0),
$$
$$
u^\prime_0(0)=u_0(h_0)=v^\prime_0(0)=v_0(h_0)=0.
$$

According to the classic Lotka-Volterra type predator-prey theory, there exist a ``paradox of enrichment" stating that enriching the prey's environment always leads to an unstable predator-prey system, and a ``biological control paradox" which states that a low and stable prey equilibrium density does not exist. These two situations are inconsistent with the real world. In many situations, especially when predators have to search, share and compete for food, many mathematicians and biologists have confirmed that a ratio-dependent predator-prey model is more reasonable than the prey-dependent model (see \cite{ratio-dependent2},\cite{19},\cite{parodoxofenrichment},\cite{ratio-dependent1},\cite{biologicalcontrolparadox}).

The equation $h^\prime (t)=-\mu( u_x+\rho v_x)$ governing the free boundary, is a special case of two-phase Stefan condition. Here, we assume that the expanding front propagates at a rate that is proportional to the magnitudes of the prey's and predator's populations gradients.
In fact, both prey and predator have a tendency to move outward from some unknown boundary (free boundary) constantly. Suppose that the predator only lives on this prey as a result of the features of partial eclipse and picky eaters and the restraint of external environment. In order to survive, the predator should follow the same trajectory as prey. Thus, they roughly are consistent in move curve (free boundary). Moreover, we can use this model to study the following two common phenomenons:
(i) the effect of controlling pest species (prey) by introducing natural enemy (predator);
(ii) the impact of a new or invasive species (predator) on native species (prey).

The Stefan condition arises from the study of the melting of ice in water \cite{L1971The}. Later, this condition is widely applied to other problems. For example, it was applied to the model of wound healing \cite{15} and oxygen in the muscle \cite{Du2001LOGISTIC}. For population models, Du et al. \cite{Du6},\cite{Du2},\cite{Du3},\cite{Du1},\cite{Du4},\cite{2012Du} have studied a series of nonlinear diffusion problems with free boundary on the one-phase Stefan condition, addressed many critical problems such as the long time behavior of species, the conditions for spreading and vanishing and the asymptotic spreading speed of the front. In particular, if the nonlinear term is a general monostable type, then a spreading-vanishing dichotomy stands. Wang et al. have investigated a succession of free boundary problems on diverse Stefan conditions of multispecies model and get lots of useful conclusions (see \cite{wang11},\cite{wang10},\cite{wang9},\cite{wang12},\cite{wang7},\cite{wang6},\cite{wang8},\cite{Wangzhao2014a}).

In \cite{wang11}, Wang studied the same free boundary problem but for the classical Lotka-Volterra type predator-prey model. A spreading-vanishing dichotomy was proved and the long time behavior of solution and criteria for spreading and vanishing were obtained. Moreover, when spreading successfully, an upper bound of the spreading speed was provided. \cite{liu2020free} studied a ratio-dependent predator-prey problem with a different free boundary in which the spreading front was only caused by prey. The author studied the spreading behaviors of the two species and provided an accurate limit of the spreading speed as $t\rightarrow\infty$.

In this paper, we mainly research problem (\ref{Q}) and understand the asymptotic behaviors of prey and predator via such a free boundary causing by both prey and predator. We always assume that $(u,v,h)$ is the solution to problem (\ref{Q}) in this paper. For the global existence, uniqueness and estimates of positive solution $(u,v,h)$ , similar to the proofs of Theorem 2.1, Lemma 2.1 and Theorem 2.2 in \cite{liu2020free}, we can prove a theorem as follows.
\begin{theorem}
For any $0<\alpha<1$, there exists $T>0$ such that 
$$
(u,v,h)\in[\mathcal{C}^{\frac{1+\alpha}{2},1+\alpha}(\overline D_T)]^2\times\mathcal C^{1+\frac{\alpha}{2}}([0,T]),
$$ 
where
$$
D_T=\{(t,x)\in\mathbb{R}^2:t\in(0,T],x\in(0,h(t))\}.
$$
Furthermore, for $(t,x)\in(0,\infty)\times (0,h(t))$, there exists a positive constant $M$ such that
$$
0<u(t,x),v(t,x),h^{\prime}(t)\leq M.
$$
\end{theorem}

The organization of this paper is as follows. In section 2, we provide some compare principles which prepare for the following research. In Section 3, we studies the waves of finite length to construct a lower solution and obtain a spreading-vanishing dichotomy. Section 4 is devoted to the study of criteria governing spreading and vanishing. In Section 5, an estimate of asymptotic spreading speed is obtained. Section 6 gives a brief discussion.

\section{Compare principles}
In this section we provide some compare principles with free boundaries which are critical to the subsequent research. 
\begin{lemma}\label{lm2.1}
Define $\Omega=\{(t,x):t>0,0<x<\overline h(t)\}$. Let $\overline u,~\overline v\in\mathcal{C}(\overline\Omega)\bigcap\mathcal{C}^{1,2}(\Omega)$, $\overline h\in\mathcal{C}^1([0,\infty))$ and $\overline h(t)>0$ for $t\geq0$. If $(\overline u,\overline v,\overline h)$ satisfies
$$
\begin{cases}
\overline u_t-\overline u_{xx}\geq\lambda\overline u-\overline u^2,    &t>0 ,~0< x<\overline h(t),\\
\overline v_t -d\overline v_{xx}\geq(\nu+c)\overline v-\overline v^2,       &t>0 ,~0<x<\overline h(t),\\
\overline u_x(t,0)\leq0,~\overline v_x(t,0)\leq0,                          &t>0,\\
\overline u((t,\overline h(t)))=\overline v(t,\overline h(t))=0,    &t\geq0,\\
\overline h^\prime (t)\geq-\mu[ \overline u_x(t,\overline h(t))+\rho \overline v_x(t,\overline h(t))], &t>0,\\
\overline u(0,x)\geq u_0(x),~ \overline v(0,x)\geq v_0(x),        &0\leq x\leq \overline h_0,\\
\overline h(0)\geq h_0,
\end{cases}
$$
then we get
$$
u\leq\overline u,~v\leq\overline v~on~D,~h(t)\leq\overline h(t)~for~t\geq0,
$$
where $D:=\{(t,x):t\geq0,0\leq x\leq h(t)\}$.
\end{lemma}

Define $\Omega_1=\{(t,x):t>0,0<x<\underline h(t)\}$ and let $\underline h\in\mathcal{C}^1([0,\infty))$ with $0<\underline h(0)<h_0$. Similar to the above Lemma \ref{lm2.1}, we present a lower solution of $(u,h)$ and $(v,h)$, respecively.

\begin{lemma}\label{lm2.2}
Let $\underline u\in\mathcal{C}(\overline\Omega_1)\bigcap\mathcal{C}^{1,2}(\Omega_1)$. If $(\underline u,\underline h)$ satisfies
$$
\begin{cases}
\underline u_t-\underline u_{xx}\leq(\lambda-\frac{b}{m})\underline u-\underline u^2,~~~&t>0,~0<x<\underline h(t),\\
\underline u_x(t,0)=\underline u(t,\underline h(t))=0, &t>0,\\
\underline h^{\prime}(t)\leq-\mu\underline u_x(t,\underline h(t)),&t>0,\\
0\leq\underline u(0,x)\leq u_0(x),&0\leq x\leq\underline h(0),\\
\underline h(0)\leq h(0),
\end{cases}
$$
then we have
$$
h(t)\geq\underline h(t),t\geq0;~u(t,x)\geq\underline u(t,x)~on~\overline \Omega_1.
$$
\end{lemma}

\begin{lemma}\label{lm2.3}
Let $\underline  v\in\mathcal{C}(\overline \Omega_1)\bigcap\mathcal{C}^{1,2}(\Omega_1)$. If $(\underline  v,\underline h)$ satisfies
$$
\begin{cases}
\underline  v_t-\underline  v_{xx}\leq\nu\underline  v-\underline  v^2,~~~&t>0,~0<x<\underline h(t),\\
\underline  v_x(t,0)=\underline  v(t,\underline h(t))=0, &t>0,\\
\underline h^{\prime}(t)\leq-\mu\rho\underline  v_x(t,\underline h(t)),&t>0,\\
0\leq\underline  v(0,x)\leq v_0(x),&0\leq x\leq\underline h(0),\\
\underline h(0)\leq h_0,
\end{cases}
$$
then we have
$$
h(t)\geq\underline h(t),t\geq0;~v(t,x)\geq\underline  v(t,x)~on~\overline  \Omega_1.
$$
\end{lemma}

\begin{remark}\label{rm3}
We also can define an upper solution to $(u,h)$ and $(v,h)$ by reversing all the inequalities in Lemma \ref{lm2.2} and \ref{lm2.3}.
\end{remark}

\section{Waves of finite length and the spreading-vanishing dichotomy}

In this section we study the long time behavior of $(u,v)$. Since $h(t)$ is monotonic increasing then either $h(t)<\infty$ (vanishing case) or $h(t)\rightarrow\infty$ (spreading case) as $t\rightarrow\infty$.

\subsection{Spreading case ($h_{\infty}=\infty)$}

Assume that $h_{\infty}=\infty$, then (\ref{Q}) becomes
\begin{equation}
\begin{cases}
u_t-u_{xx}=\lambda u-u^2-\frac{buv}{u+mv},~~~&t>0,~x>0,\\
v_t-dv_{xx}=\nu v-v^2+\frac{cuv}{u+mv},     &t>0,~x>0,\\
u_x(t,0)=v_x(t,0)=0,                        &t>0,\\
u(0,x)=u_0(x),v(0,x)=v_0(x),                 &x\geq0,
\end{cases}
\end{equation}
and its stationary problem is
\begin{equation}\label{31}
\begin{cases}
-u_{xx}=\lambda u-u^2-\frac{buv}{u+mv},~~~&x>0,\\
-dv_{xx}=\nu v-v^2+\frac{cuv}{u+mv},      &x>0,\\
u(x)=u_0(x),v(x)=v_0(x),                 &x\geq0.
\end{cases}
\end{equation}

In the same way as the proof of Theorem 3.2 in \cite{liu2020free}, we can prove the following theorem.
\begin{theorem}\label{th2.3}
Assume $h_{\infty}=\infty$.

{\upshape (i)}If $m\lambda>b$ then the solution $(u,v)$ satisfies

$$
\underline u\leq\liminf\limits_{t\rightarrow\infty}u(t,x)\leq\limsup\limits_{t\rightarrow\infty}u(t,x)\leq\overline u,
$$
$$
\underline v\leq\liminf\limits_{t\rightarrow\infty}v(t,x)\leq\limsup\limits_{t\rightarrow\infty}v(t,x)\leq\overline v
$$
uniformly on the compact subset of $[0,\infty)$, where $\overline u$, $\underline u$, $\overline v$, $\underline v$ are determined by
$$
\lambda-\underline u-\frac{b\overline v}{\underline u+m\overline v}=0,~~\lambda-\overline u-\frac{b\underline v}{\overline u+m\underline v}=0,
$$
$$
\nu-\overline v+\frac{c\overline u}{\overline u+m\overline v}=0,~\nu-\underline v+\frac{c\underline u}{\underline u+m\underline v}=0.
$$
{\upshape(ii)}If $0<m\lambda-b<b\nu/c$, then
$$\lim\limits_{t\rightarrow\infty}u(t,x)=u^*:=\frac{A+\sqrt{\Delta_1}}{2(b +c m^2)},~\lim\limits_{t\rightarrow\infty}v(t,x)=v^*:=\frac{u^*(\lambda-u^*)}{b-m(\lambda-u^*)},$$
where $A=\lambda(2cm^2+b)-mb(\nu+2c)$,
$\Delta_1=A^2+4(b+cm^2)[(b(\nu+c)-mc\lambda)](m\lambda-b)$. Moreover, $(u^*,v^*)$ is the stationary solution of (\ref{31}).
\end{theorem}

\subsection{vanishing case}
In this section, we want to study the vanishing case. In order to get sufficient conditions of vanishing, we will construct a suitable lower solution to (\ref{Q}) with respect to $v$ by a phase plane analysis of the equation (\ref{17}).
\subsubsection{Waves of finite length}
In this section, we mainly study the solution $(s,q(z))$ of the following problem for $Z\in(0,\infty)$
\begin{equation}\label{17}
\begin{cases}
dq^{\prime\prime}-sq^{\prime}+f(q)=0,~&z\in[0,Z],\\
q(0)=0,~q^{\prime}(Z)=0,~q(z)>0,            &z\in[0,Z],
\end{cases}
\end{equation}
where $f(q):=\nu q-q^2+\frac{cuq}{u+mq}$.
Define $q^{\prime}=dq/dz$, then (\ref{17}) is equivalent to
\begin{equation}\label{14}
\begin{cases}
q^{\prime}=p,\\
dp^{\prime}=sp-f(q),
\end{cases}
\end{equation}
or
\begin{equation}\label{15}
d\cdot\frac{dp}{dq}=s-\frac{f(q)}{p},~when ~p\neq0.
\end{equation}
For each $s\geq0$ and $\eta>0$, we denote $p^s(q;\eta)$ as the unique solution of (\ref{15}) with initial condition $p^s(q)|_{q=0}=\eta$, where $\eta>0$. We mainly discuss the cases $s=0$ and a small $s>0$.

When $s=0$. A simple calculation deduces that
\begin{equation}\label{16}
p^0(q;\eta)=\sqrt{\eta^2-\frac{2}{d}\int_0^q f(\tau)d\tau},~~q\in[0,q^{\eta}),
\end{equation}
where $q^{\eta}$ is given by
\begin{equation}\label{18}
\eta^2=\frac{2}{d}\int_0^{q^{\eta}} f(\tau)d\tau.
\end{equation}
Denote $\theta:=v^*$, where $v^*$ is defined by Theorem \ref{th2.3}. It follows that $q^{\eta}<\theta~(<\nu+c)$ if and only if $0<\eta<\eta^*$, where
$$
\eta^*=\sqrt{\frac{2}{d}\int_0^{\theta} f(\tau)d\tau}.
$$
Furthermore, $q^{\eta}$ is strictly increasing in $\eta\in(0,\eta^*)$ and $q^{\eta}\rightarrow0$ as $\eta\rightarrow0$.

\begin{figure}[htbp]
\centering
{
\begin{minipage}[t]{0.3\linewidth}
\centering
\includegraphics[width=1.5in]{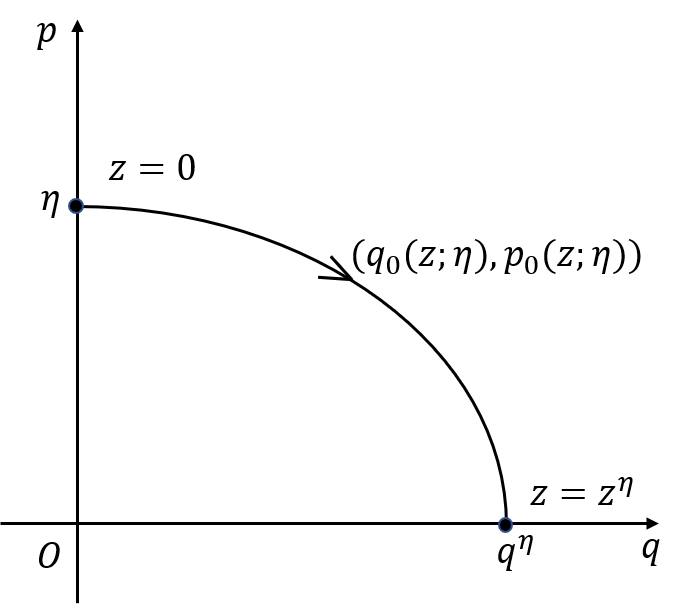}
\caption{$s=0$}
\end{minipage}
}
{
\begin{minipage}[t]{0.5\linewidth}
\centering
\includegraphics[width=2.1in]{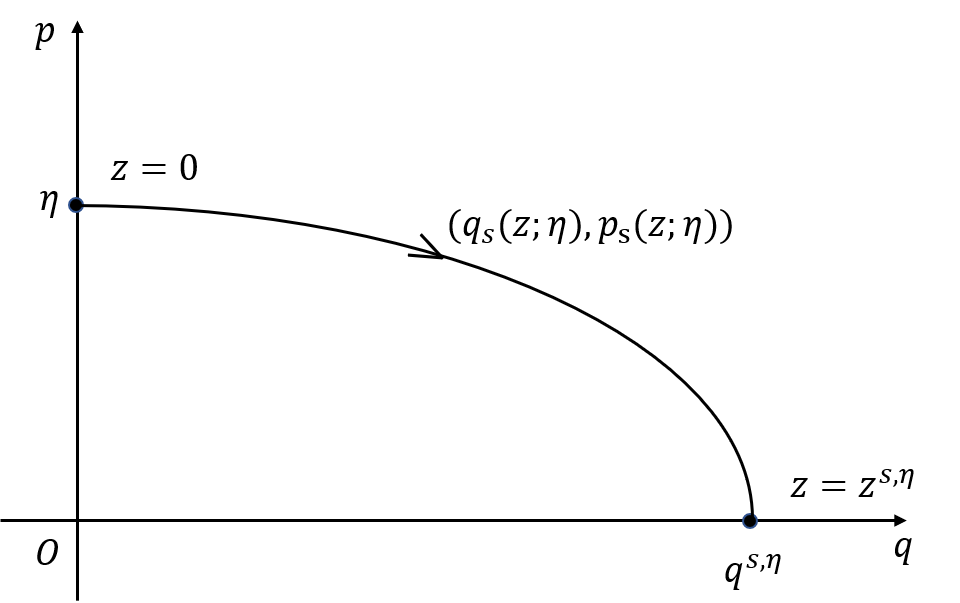}
\caption{A small $s>0$}
\end{minipage}
}
\end{figure}

The positive solution $p^0(q;\eta)$ of (\ref{14}) corresponds to a trajectory
$(q_0(z;\eta)$, $p_0(z;\eta))$ (with $s=0$) that passes through $(0,\eta)$ at $z=0$ and approaches $(q^{\eta},0)$ as $z$ goes to $z^{\eta}$ (see Figure 1). It follows from (\ref{14}) with $s=0$ and (\ref{16}) and (\ref{18}) that
$$
z=\int_0^{q_0(z;\eta)}\frac{dr}{\sqrt{\frac{2}{d}\int_r^{q^{\eta}}f(\tau)d\tau}}.
$$
So
$$
z^{\eta}=\int_0^{q^{\eta}}\frac{dr}{\sqrt{\frac{2}{d}\int_r^{q^{\eta}}f(\tau)d\tau}}.
$$
Recall that $q^{\eta}\rightarrow0$ as $\eta\rightarrow0$ and so
$$
z^{\eta}=\int_0^{q^{\eta}}\frac{\sqrt{d}+o(1)}{\sqrt{f^{\prime}(0)((q^{\eta})^2-r^2)}}dr
=\frac{\pi}{2}\sqrt{\frac{d}{f^{\prime}(0)}}+o(1).
$$
Define $$Z^*:=\frac{\pi}{2}\sqrt{\frac{d}{f^{\prime}(0)}}.$$

According to the above discussions, we have the following result.
\begin{lemma}\label{lm3.1}
If $Z>Z^*$, then the elliptic boundary value problem
\begin{equation}\label{19}
\begin{cases}
dv_{xx}+f(v)=0,~x\in(0,Z),\\
v^{\prime}(0)=v(Z)=0
\end{cases}
\end{equation}
\end{lemma}
has at least one positive solution $v_{Z}$.

\begin{proof}
Since $Z> Z^*$ there exists $\eta_*\in(0,\eta^*)$ and correspondingly $q_*:=q^{\eta_*}\in(0,\theta)$ such that $z_*:=z^{\eta_*}\in(Z^*,Z)$. Let $(q(z),p(z))$ be the trajectory of (\ref{14}) (with $s=0$) that connects $(0,\eta_*)$ at $z=0$ and $(q_*,0)$ as $z$ goes to $z_*$. Then $q(z)$ satisfies
$$
\begin{cases}
dq^{\prime\prime}+f(q)=0,~z\in(0,z_*),\\
q(0)=q^{\prime}(z_*)=0.
\end{cases}
$$
Define
$$
\underline v(x):=
\begin{cases}
q(-x+z_*),~&x\in(0,z_*],\\
0,         &x\in(z_*,Z].
\end{cases}
$$
Then $\underline v$ is a (weak) lower solution of (\ref{19}). On the other hand, a big enough constant $C\gg\nu+c$ is an upper solution of (\ref{19}). Therefore, (\ref{19}) at least one positive solution by the standard upper-lower solution argument.
\end{proof}

\begin{remark}\label{rm1}
The positive solution $v_Z$ of (\ref{19}) corresponds to a trajectory $(q(z),p(z)):=(v_Z(Z-z),-v^{\prime}_Z(Z-z))$ (with $s=0$) passing through $(0,\eta):=(0,-v^{\prime}_Z(Z))$ at $z=0$ and approaching $(q^{\eta},0):=(v_Z(0),0)$ as $z$ goes to $Z$.
\end{remark}

Now we study (\ref{14}) for small $s>0$ as a perturbation of the case $s=0$. For some small $s>0$, (\ref{15}) with initial data $p^s(q)|_{q=0}=\eta\in(0,\eta^*)$ has a solution $p^s(q;\eta)$ defined on $[0,q^{s,\eta}]$ for some $q^{s,\eta}>q^{\eta}$. Let $(q_s(z;\eta),p_s(z;\eta))$ be the trajectory of (\ref{14}) (with small $s>0$) that pass through $(0,\eta)$ at $z=0$ and approaches $(q^{s,\eta},0)$ as $z$ goes to $z^{s,\eta}$ (See Figure 2). Furthermore, we have the following results.

\begin{lemma}\label{lm1}
Fix $\eta\in(0,\eta^*)$. For any $\varepsilon>0$, there exists some small $\delta>0$ such that

(\upshape{i})if $s\in(0,\delta)$, then $q^{s,\eta}\in(q^{\eta},q^{\eta}+\varepsilon)$ and $z^{s,\eta}\in(z^{\eta}-\varepsilon,z^{\eta}+\varepsilon)$ ;

(\upshape{ii})$p^0(q;\eta)\leq p^s(q;\eta)\leq p^0(q;\eta)+\varepsilon$ for $q\in[0,q^{\eta}]$;

(\upshape{iii})$q_0(z;\eta)\leq q_s(z;\eta)\leq q_0(z;\eta)+\varepsilon$ for $z\in[0,\min\{z^{\eta},z^{s,\eta}\}]$.
\end{lemma}

\subsubsection{Vanishing case}
In order to discuss the long time behavior of $(u,v)$, we first give two important propositions.
\begin{proposition}\label{th2.1}
If $h_{\infty}<\infty$, then there exists a positive constant $M$ such that
$$
\|u(t,\cdot),v(t,\cdot)\|_{\mathcal C^1[0,h(t)]}\leq M,~\forall t>1.
$$
and
$$
\lim\limits_{t\rightarrow\infty}h^{\prime}(t)=0.
$$
\end{proposition}

\begin{proof}
Similar to the proof of Theorem 4.1 of \cite{wang6}, we omit it.
\end{proof}

\begin{proposition}\label{p1}{\upshape(\cite{wang11})}
Let $d$, $\theta$, $\beta$, $g_0$, $C$ be positive constants. Suppose that $w\in\mathcal C ^{\frac{1+\alpha}{2},1+\alpha}([0,\infty)\times[0,g(t)])$ and $g\in\mathcal C^{1+\frac{\alpha}{2}}([0,\infty))$ for some $\alpha>0$ and satisfy $w(t,x)>0$, $g(t)>0$ for all $0\leq t<\infty$ and $0<x<g(t)$. Assume that $w_0\in\mathcal C^2([0,g_0])$ and satisfies $w_0^{\prime}(0)=0$, $w_0(g_0)=0$ and $w_0(x)>0$ in $(0,g_0)$. Furthermore, suppose that $$\lim\limits_{t\rightarrow\infty}g(t)=g_{\infty}<\infty,~ \lim\limits_{t\rightarrow\infty}g^{\prime}(t)=0,~\|w(t,\cdot)\|_{\mathcal C[0,g(t)]}\leq M,~~\forall t>1.
$$
If $(w,g)$ satisfies
$$
\begin{cases}
w_t-dw_{xx}\geq w(C-w),~~&t>0,0<x<g(t),\\
w_x=0,                        &t>0,x=0,\\
w=0,~g^{\prime}(t)\geq-\beta w_x,&t>0,x=g(t),\\
w(0,x)=w_0(x),       &0\leq x\leq g_0,\\
g(0)=0,
\end{cases}
$$
then $$\lim\limits_{t\rightarrow\infty}\max\limits_{0\leq x\leq g(x)}w(t,x)=0.$$
\end{proposition}

\begin{lemma}\label{th2.2}
Let $(u,v,h)$ be solution of the problem (\ref{Q}). If $h_{\infty}<\infty$, then
\begin{equation}\label{1}
\lim\limits_{t\rightarrow\infty}\|u(t,\cdot),v(t,\cdot)\|_{\mathcal C([0,h(t)])}=0.
\end{equation}
Moreover,
\begin{equation}\label{2}
h_{\infty}\leq\frac{\pi}{2}\min\{\sqrt{{m}/{(m\lambda-b)}},\sqrt{{d}/\nu}\}.
\end{equation}
\end{lemma}

\begin{proof}
We first prove (\ref{1}). Since $(u,h)$ satisfies
$$
\begin{cases}
u_t-u_{xx}\geq u(\lambda-b/m-u),~~&t>0,~0<x<h(t),\\
u_x=0,                            &t>0,x=0,\\
u=0,~h^{\prime}(t)\geq-\mu u_x,  &t>0,x=h(t),\\
u(0,x)=u_0(x),         &0\leq x\leq h_0,\\
h(0)=h_0.
\end{cases}
$$
By Proposition \ref{th2.1} and Proposition \ref{p1} we have $\lim\limits_{t\rightarrow\infty}\|u(t,\cdot)\|_{\mathcal C([0,h(t)])}=0$.
On the other hand, $(v,h)$ satisfies
$$
\begin{cases}
v_t-dv_{xx}\geq v(\nu-v),~~&t>0,~0<x<h(t),\\
v_x=0,                            &t>0,x=0,\\
v=0,~h^{\prime}(t)\geq-\mu\rho v_x,  &t>0,x=h(t),\\
v(0,x)=v_0(x),~h(0)=h_0,         &0\leq x\leq h_0.
\end{cases}
$$
Similarity, we conclude that $\lim\limits_{t\rightarrow\infty}\|v(t,\cdot)\|_{\mathcal C([0,h(t)])}=0$.

Now we proof (\ref{2}) and firstly assert that $h_{\infty}\leq\frac{\pi}{2}\sqrt{{m}/({m\lambda-b})}$. Otherwise there exists $\tau\gg1$ such that
$$
h(\tau)>\max\{h_0,\frac{\pi}{2}\sqrt{{m}/{(m\lambda-b)}}\}.
$$
Let $l=h(\tau)$, then $l>\frac{\pi}{2}\sqrt{{m}/{(m\lambda-b)}}$. Suppose $w(t,x)$ be the unique solution of the following problem
$$
\left\{
\begin{array}{ll}
w_t-w_{xx}=w(\lambda-{b}/{m}-w),    &t>\tau,~0<x<l,\\
w_x(t,0)=w(t,l)=0,  &t>\tau,\\
w(\tau,x)=u(\tau,x),&0\leq x<l.
\end{array}
\right.
$$
By using the compare principle, we have
$$
w(t,x)\leq u(t,x),~~~t\geq\tau,~0\leq x\leq l.
$$
In view of $\lambda-{b}/{m}>(\frac{\pi}{2l})^2$, it is well known that $w(t,x)\rightarrow w^*(x)$ as $t\rightarrow\infty$ uniformly in compact subset of $[0,l)$, where $w^*$ is the unique positive solution of
$$
\left\{
\begin{array}{ll}
-w^*_{xx}=w^*(\lambda-{b}/{m}-w^*),    &0<x<l,\\
w^*_x(0)=w^*(l)=0.
\end{array}
\right.
$$
Thus,$\varliminf\limits_{t\rightarrow\infty}u(t,x)\geq\lim\limits_{t\rightarrow\infty}w(t,x)=w^*(x)>0$. This contradicts (\ref{1}). Similarity, we have $h_{\infty}\leq\frac{\pi}{2}\sqrt{{d}/{\nu}}$. The proof of (\ref{2}) is finished.
\end{proof}

The following lemma is one of the most important results in this section, which gives a more precise upper bound of $h_\infty$ when vanishing happens.
\begin{lemma}\label{th3.1}
If $h_{\infty}<\infty$, then
\begin{equation}\label{21}
h_{\infty}\leq Z^*:= \frac{\pi}{2}\sqrt{\frac{d}{f^{\prime}(0)}}.
\end{equation}
That is to say, $h_{\infty}\leq\frac{\pi}{2}\sqrt{{d}/({\nu+c})}$.
\end{lemma}

\begin{proof}
Otherwise we can find a $t_0>0$ such that $h(t_0)>Z^*$. For a small $s<\mu\rho\eta$, we want to use $q_s(z;\eta)$ to construct a lower solution of (\ref{Q}). Define
$$
k(t):=z^{s,\eta}+st,~where~z^{s,\eta}\leq Z^*,
$$
$$
w(t,x):=
\begin{cases}
q_s(z^{s,\eta};\eta),~&x\in[0,st],\\
q_s(k(t)-x;\eta),     &x\in[st,k(t)].
\end{cases}
$$
Then
$w_t\leq w_{xx}+f(w)$ and $w_x(t,0)=w(t,k(t))=0$ for $t>0$, $x\in(0,k(t))$ .
Moreover,
$$
 k(0)=z^{s,\eta}\leq Z^*<h(t_0),
$$
$$
k^{\prime}(t)=s<\mu\rho\eta=-\mu \rho w_x(t,k(t)).
$$

Now we assert that
\begin{equation}\label{20}
v(t_0,x)>w(0,x):=q_s(z^{s,\eta}-x;\eta),~x\in[0,z^{s,\eta}]
\end{equation}
holds. According to Lemma \ref{lm3.1}, problem (\ref{19}) with right boundary $h(t_0)$ replacing $Z$ has a positive solution $v_{h(t_0)}=:v_{t_0}$, which is a stationary solution. By the standard comparison principle we have
$$
v(t,x)>v_{t_0}(x),~x\in[0,h(t_0)],~t>0.
$$
So there exists a small $\varepsilon>0$ such that for $t\geq t_0$ we have
$$
v(t,x)>v_{t_0}(x)+\varepsilon,~ x\in[0,h(t_0)]
$$
and
$$
v(t,x)>v_{t_0}(0)+\varepsilon,~x\in[0,\varepsilon].
$$
By Remark \ref{rm1}, Lemma \ref{lm1}, we can find a small $s>0$ such that
$$
q_s(z^{s,\eta}-z;\eta)<q_0(z^{s,\eta}-z;\eta)+\varepsilon/2<q_0(z^{\eta}-z;\eta)+\varepsilon,~z\in[\varepsilon,z^{s,\eta}].
$$
Due to the property that $q_0(z;\eta)$ and $q_s(z;\eta)$ increases monotonically with respect to $z$, we find
$$
q_0(z^\eta-z;\eta)<q_0(h(t_0)-z;\eta)=v_{t_0}(z),~z\in[\varepsilon,z^{s,\eta}],
$$
$$
q_s(z^{s,\eta}-z;\eta)<q^{s,\eta}<q^{\eta}+\varepsilon=v_{t_0}(0)+\varepsilon,~z\in(0,\varepsilon].
$$
Thus, we have
$$
v(t,x)>q_s(z^{s,\eta}-x;\eta),~t\geq t_0,x\in[0,z^{s,\eta}].
$$
Let $t=t_0$, then (\ref{20}) is proved.

Thus, applying Lemma \ref{lm2.3} we obtain that
$$
h(t+t_0)\geq k(t),~v(t+t_0,x)\geq w(t,x),~t>0,x\in[0,k(t)],
$$
which implies that $h_{\infty}=\infty$. By Theorem \ref{th2.3} we have $\lim\limits_{t\rightarrow\infty}v(t,x)=v^*>0$, which contradicts to Theorem \ref{th2.2}. Thus (\ref{21}) is true.
\end{proof}

Combining Lemma \ref{th2.2} with Lemma \ref{th3.1}, we have the following theorem directly.

\begin{theorem}\label{th3.2}
Define
$$
\Lambda:=\frac{\pi}{2}\min\left\{\sqrt{ \frac{m}{m\lambda-b}},\sqrt{\frac{d}{\nu+c}}\right\}.
$$
If $h_\infty<\infty$, then $h_{\infty}\leq\Lambda$.
\end{theorem}

\begin{remark}
Theorem \ref{th3.2} shows that if the prey and predator cannot spread into infinity, then they will never break through $\Lambda$ and  will vanish eventually.
\end{remark}

\section{The criteria governing spreading and vanishing}
In this section, we study the criteria of spreading and vanishing for the problem (\ref{Q}). Recall that $h^{\prime}(t)>0$ for $t>0$, then an important result is obtained directly by Theorem \ref{th3.2} as follows.
\begin{theorem}
If $h_0\geq\Lambda$, then $h_{\infty}=\infty$.
\end{theorem}

Next we mainly discuss the case $h_0<\Lambda$.

\begin{lemma}\label{lm4.1}
Suppose $h_0<\Lambda$. If
$$
\mu\geq\mu^0:=\min\{\mu^*,\mu^{**}\},
$$
where
$$
\mu^*:=\max\left\{1,\frac{m\|u_0\|_{\infty}}{m\lambda-b}\right\}(\frac{\pi}{2}\sqrt{\frac{m}{m\lambda-b}}-h_0)(\int_0^{h_0}u_0(x)dx)^{-1},
$$
$$
\mu^{**}:=\max\left\{1,\frac{\|v_0\|_{\infty}}{\nu}\right\}\frac{d}{\nu}(\frac{\pi}{2}\sqrt{\frac{d}{\nu+c}}-h_0)(\int_0^{h_0}v_0(x)dx)^{-1},
$$
then $h_{\infty}=\infty$.
\end{lemma}

\begin{proof}
We firstly consider the following auxiliary problem
$$
\begin{cases}
\underline u_t-\underline u_{xx}=(\lambda-\frac{b}{m})\underline u-{\underline u}^2,~&t>0,0<x<\underline h(t),\\
\underline u_x(t,0)=\underline u(t,h(t))=0,&t>0\\
\underline h^{\prime}(t)=-\mu\underline u_x(t,h(t)),&t>0,\\
\underline u(0,x)=u_0(x),&0\leq x\leq h_0,\\
\underline h(0)=h_0.
\end{cases}
$$
We just discuss the case $\|u_0\|_{\infty}\leq\lambda-\frac{b}{m}$. If $\|u_0\|_{\infty}>\lambda-\frac{b}{m}$ then we can take $\underline u_0=\frac{(m\lambda-b)u_0(x)}{m\|u_0\|_{\infty}}$.
Direct calculations give
$$
\begin{aligned}
\frac{d}{dt}\int_0^{\underline h(t)}\underline u(t,x)dx
&=\int_0^{\underline h(t)}\underline u_t(t,x)dx+\underline h^{\prime}(t)\underline u(t,\underline h(t))\\
&=\int_0^{\underline h(t)}u_{xx}dx+\int_0^{\underline h(t)}(\lambda-\frac{b}{m})\underline u-\underline u^2 dx\\
&=-\frac{\underline h^{\prime}(t)}{\mu}+\int_0^{\underline h(t)}(\lambda-\frac{b}{m})\underline u-\underline u^2 dx.
\end{aligned}
$$
Then we integrate $0$ to $t$ and get
$$
\begin{aligned}
\int_0^{\underline h(t)}\underline u(t,x)dx
&=(\int_0^{\underline h_0}u_0(x)dx+\frac{h_0-\underline h(t)}{\mu})+\int_0^t\int_0^{\underline h(s)}(\lambda-\frac{b}{m})\underline u-\underline u^2 dxds\\
&:=\uppercase\expandafter{\romannumeral1}+\uppercase\expandafter{\romannumeral2}.
\end{aligned}
$$
Notice that $0<\underline u(t,x)<\lambda-\frac{b}{m}$ for all $t>0$ and $x\in[0,\underline h(t)]$ and so we have $\uppercase\expandafter{\romannumeral2}>0$ for $t>0$.

Assume that $h_{\infty}\neq\infty$. By Theorem \ref{th3.2} and Lemma \ref{lm2.2} we have $\underline h_{\infty}:=\lim\limits_{t\rightarrow\infty}\underline h(t)\leq\frac{\pi}{2}\sqrt{\frac {m}{m\lambda-b}}$ and $\lim\limits_{t\rightarrow\infty}\|\underline u(t,\cdot)\|_{\mathcal{C}([0,h(t)])}=0$. Thus $\int_0^{\underline h(t)}\underline u(t,x)dx\rightarrow0$ leads to  $\uppercase\expandafter{\romannumeral1}<0$
as $t\rightarrow\infty$, which is a contradiction to our assumption $\mu\geq\mu^*$. Thus, it turns out that if $\mu>\mu^*$, then $h_{\infty}=\infty$.

We now consider the following auxiliary problem
$$
\begin{cases}
\underline v_t-d\underline v_{xx}=\nu\underline v-{\underline v}^2,~&t>0,0<x<\underline h(t),\\
\underline v_x(t,0)=\underline v(t,h(t))=0,&t>0,\\
\underline h^{\prime}(t)=-\mu\rho\underline v_x(t,h(t)),&t>0,\\
\underline v(0,x)=u_0(x),&0\leq x\leq h_0,\\
\underline h(0)=h_0.
\end{cases}
$$
Similar to the above discussion and noticing that $h_0\leq\max\{\frac{\pi}{2}\sqrt{\frac{d}{\nu}},\frac{\pi}{2}\sqrt{\frac{d}{\nu+c}}\}$, we obtain that if
$$
\begin{aligned}
\mu&\geq\max\{1,\frac{\|v_0\|_{\infty}}{\nu}\}\cdot\frac{d}{\nu}\cdot(\min\{\frac{\pi}{2}\sqrt{\frac{d}{\nu}},\frac{\pi}{2}\sqrt{\frac{d}{\nu+c}}\}-h_0)(\int_0^{h_0}v_0(x)dx)^{-1}\\
   &=\mu^{**}
\end{aligned}
$$
then $\underline h_{\infty}=\infty$. By Lemma \ref{lm2.3} we have $h_{\infty}=\infty$. The proof is finished.
\end{proof}

\begin{lemma}\label{lm4.2}
Assume $h_0<\Lambda$. There exists $\mu_0>0$ depending on $u_0$ and $v_0$ such that $h_{\infty}<\infty$ if $\mu\leq\mu_0$.
\end{lemma}

\begin{proof}
We will use Lemma \ref{lm2.1} and construct a suitable upper solution of (\ref{Q}) which inspired by \cite{Du1} and \cite{wang11} to derive the desired conclusion. Define
$$
\sigma(t)=h_0(1+\delta-\frac{\delta}{2}e^{-\beta t}),~t\geq0;~V(y)=\cos(\frac{\pi y}{2}),~0\leq y\leq1,
$$
$$
\overline u(t,x)=\overline v(t,x)=Me^{-\beta t}V\left(\frac{x}{\sigma(t)}\right),~t\geq0,~0\leq x\leq\sigma(t),
$$
where $\beta$, $\delta$ and $M$ are positive constants to be chosen later.

Obviously, we have
$$
\sigma(0)=h_0(1+\frac{\delta}{2})>h_0,~h_0(1+\frac{\delta}{2})\leq\sigma(t)\leq h_0(1+\delta),
$$
$$
\overline u_x(t,0)=\overline u(t,\sigma(t))=\overline v_x(t,0)=\overline v(t,\sigma(t))=0,~\forall t\geq0.
$$
Let $M\gg1$ such that $\overline u(0,x)\geq u_0(x)$, $\overline v(0,x)\geq v_0(x)$ for $x\in[0,h_0]$
and take $\beta=\frac{1}{2}(\frac{\pi}{2})^2h^{-2}_0(1+\delta)^{-2}-\frac{1}{2}\max\{\lambda,\nu+c\}$.
Then direct computations yield
$$
\begin{aligned}
 &\overline u_t-\overline u_{xx}-\overline u(\lambda-\overline u)\\
=&\overline u\left(-\beta+\frac{\pi}{2}x\sigma^{-2}\sigma^{\prime}\tan(\frac{\pi }{2}\frac{x}{\sigma(t)}) +(\frac{\pi}{2})^2\sigma^{-2}-\lambda+\overline u\right)\\
\geq&\overline u\left(-\beta+(\frac{\pi}{2})^2\sigma^{-2}-\lambda\right)\\
>&0,~~~~~~t>0,~0\leq x\leq\sigma(t).
\end{aligned}
$$
Similarly, we have
$$
\overline v_t-\overline v_{xx}-\overline v(\nu+c-\overline v)>0,~~~~t>0,~0\leq x\leq\sigma(t).
$$
Choose $\mu_0=\frac{\delta\beta h_0^2}{2\pi M(1+\rho)}$ and then for any $0\leq\mu\leq\mu_0$ we have
$$
\sigma^{\prime}(t)+\mu(\overline u_x+\rho\overline v_{x})|_{x=\sigma(t)}=\frac{e^{-\beta t}}{2} \left({\delta\beta h_0}-\frac{\pi M\mu(1+\rho)}{\sigma(t)}\right)>0.
$$
By virtue of Lemma \ref{lm2.1}, we have $\sigma(t)\geq h(t)$. Let $t\rightarrow\infty$ and then we get $h_{\infty}\leq\sigma(\infty)=h_0(1+\delta)<\infty$. Thus the proof is finished.
\end{proof}

\begin{theorem}
Assume that $h_0<\Lambda$. Then there exist $\overline\mu\geq\underline\mu>0$, depending on $u_0$, $v_0$ and $h_0$, such that $h_{\infty}\leq\Lambda$ if $\mu\leq\underline\mu$ and $h_{\infty}=\infty$ if $\mu>\overline\mu$.
\end{theorem}

\begin{proof}
The proof is similar to that of Theorem 3.9 in \cite{Du1}. To emphasize the dependence $(u,v,h)$ on $\mu$, we write it as $(u_{\mu},v_{\mu},h_{\mu})$. Define
$$
\Sigma^*:=\{\mu>0:h_{\mu,\infty}\leq\Lambda\}~and~\overline\mu:=\sup\Sigma^*.
$$
So $h_{\mu,\infty}=\infty$ if $\mu>\overline\mu$ by Theorem \ref{th3.2}. Thus, $\Sigma^*\subset(0,\overline\mu]$. We assert that $\overline\mu\in\Sigma^*$. Otherwise, we have $h_{\overline\mu,\infty}=\infty$. Then there exists $T>0$ such that $h_{\overline\mu}(T)>\Lambda$. In view of the dependence of $(u_{\mu},v_{\mu},h_{\mu})$ on $\mu$, there exists $\varepsilon>0$ such that $h_{\mu}(T)>\Lambda$ for $\mu\in(\overline\mu-\varepsilon,\overline\mu+\varepsilon)$. Therefore, $(\overline\mu-\varepsilon,\overline\mu+\varepsilon)\bigcap\Sigma^*=\emptyset$ and $\sup\Sigma^*\leq\overline\mu-\varepsilon$, which contradicts definition of $\mu^*$. This proves the assertion $\overline\mu\in\Sigma^*$.

Let
$$
\Sigma_*:=\{\mu:\mu\geq\mu_0~such~that~h_{\mu,\infty}\leq\Lambda\}~and~\underline\mu:=\sup\Sigma_*.
$$
Then $\underline\mu\leq\overline\mu$ and $(0,\underline\mu)\subset\Sigma_*$. In the same way as above, we can prove that $\underline\mu\in\Sigma_*$. This completes the proof.
\end{proof}

\section{Asymptotic spreading speed}
In this section, we give some estimates of ${h(t)}$ to understand the asymptotic spreading speed (if spreading happens). We first introduce a vital result which can easily be deduced by Theorem 6.2 of \cite{Du4} in order to obtain an upper bound for $\limsup\limits_{t\rightarrow\infty}\frac{h(t)}{t}$.
\begin{proposition}
Let $d$, $s$, $\theta$ are positive constants. For any given $s>2\sqrt{\theta d}$, the following problem
$$
\begin{cases}
dq^{\prime\prime}-sq^{\prime}+q(\theta-q)=0,~~&z\in[0,\infty),\\
q(0)=0,~q(\infty)=\theta,\\
q(z)>0,~q^{\prime}(z)>0,&z\in[0,\infty)
\end{cases}
$$
has a unique solution.
\end{proposition}

\begin{remark}\label{rm2}
For any given $s>2\max\{\sqrt{\lambda},\sqrt{d(\nu+c)}\}$, the problem
\begin{equation}\label{23}
\begin{cases}
\phi^{\prime\prime}-s\phi^{\prime}+\phi(\lambda-\phi)=0,d\psi^{\prime\prime}-s\psi^{\prime}+\psi(\nu+c-\psi)=0&in~[0,\infty),\\
(\phi,\psi)(0)=(0,0),(\phi,\psi)(\infty)=(\lambda,\nu+c),\\
\phi>0,\psi>0,\phi^{\prime}>0,\psi^{\prime}>0,&in~[0,\infty)
\end{cases}
\end{equation}
has a unique solution $(\phi,\psi)$.
\end{remark}

\begin{theorem}\label{th5.1}
Suppose that $h_{\infty}=\infty$. Then we have
$$
\limsup\limits_{t\rightarrow\infty}\frac{h(t)}{t}\leq2\max\{\sqrt{\lambda},\sqrt{d(\nu+c)}\}.
$$
\end{theorem}
\begin{proof}
The idea of proof is inspired by \cite{Du1}. Let $s>2\max\{\sqrt{\lambda},\sqrt{d(\nu+c)}\}$ and $(\phi(\xi),\psi(\xi))$ be the solution of (\ref{23}).
Recall that $\limsup\limits_{t\rightarrow\infty}u(t,x)\leq\lambda$ and $\limsup\limits_{t\rightarrow\infty}v(t,x)\leq\nu+c$ for $x\geq0$, so for any small $\varepsilon>0$, there exists $T=T_{\varepsilon}>0$ such that
$$
u(t,x)\leq(1-\varepsilon)^{-1}\lambda,~v(t,x)\leq(1-\varepsilon)^{-1}(\nu+c),~\forall t\geq T,x\geq0.
$$
Since $\phi(\xi)\rightarrow\lambda$ and $\psi(\xi)\rightarrow\nu+c$ as $\xi\rightarrow\infty$, there exists $\xi_0>0$ such that
$$
\phi(\xi_0)>(1-\varepsilon)\lambda,~\psi(\xi_0)>(1-\varepsilon)(\nu+c).
$$

Now define
$$k(t)=(1-\varepsilon)^{-2}st+\xi_0+h(T),~t\geq0,$$
$$
\overline u(t,x)=(1-\varepsilon)^{-2}\phi(k(t)-x),~t\geq0,~0\leq x\leq\xi(t),
$$
$$
\overline v(t,x)=(1-\varepsilon)^{-2}\psi(k(t)-x),~t\geq0,~0\leq x\leq\xi(t),
$$
where
\begin{equation}\label{22}
s>\mu[\phi^{\prime}(0)+\rho\psi^{\prime}(0)].
\end{equation}
Clearly,
$$
\overline u(t,k(t))=\overline v(t,k(t))=0,
$$
$$
\overline u_x(t,0)=-(1-\varepsilon)^{-2}\phi^{\prime}(k(t))<0,~\overline v_x(t,0)=-(1-\varepsilon)^{-2}\psi^{\prime}(k(t))<0.
$$
For $x\in[0,h(T)]$,
$$
\begin{aligned}
\overline u(0,x)&=(1-\varepsilon)^{-2}\phi(\xi_0+h(T)-x)\\
                &\geq(1-\varepsilon)^{-1}\lambda\\
                &\geq u(T,x)
\end{aligned}
$$
and similarity we have $\overline v(0,x)\geq v(T,x)$. Direct calculations deduce that
$$
\begin{aligned}
 &\overline u_t-\overline u_{xx}-\overline u(\lambda-\overline u)\\
=&(1-\varepsilon)^{-2}\left[(1-\varepsilon)^{-2}s\phi^{\prime}-\phi^{\prime\prime}-\phi(\lambda-(1-\varepsilon)^{-2}\phi)\right]\\
\geq&(1-\varepsilon)^{-2}\left[s\phi^{\prime}-\phi^{\prime\prime}-\phi(\lambda-\phi)\right]\\
=&0,~~~~t>0,~0<x<\xi(t),
\end{aligned}
$$
and in the same way we get $\overline v_t-d\overline v_{xx}-\overline v(\nu+c-\overline v)\geq0$ for $t>0$, $0<x<\xi(t)$.
It follows from (\ref{22}) that
$$
\begin{aligned}
k^{\prime}(t)&=(1-\varepsilon)^{-2}s\\
             &>(1-\varepsilon)^{-2}\mu\left[\phi^{\prime}(0)+\rho\psi^{\prime}(0)\right]\\
             &=-\mu\left[\overline u_x(t,k(t))+\rho\overline v_x(t,k(t))\right].
\end{aligned}
$$
Moreover, since $h^{\prime}(t)>0$, we have $k(0)=\xi_0+h(T)>h_0$.
Therefore, by Lemma \ref{lm2.1} we have $k(t)\geq h(t+T)$. Therefore,
$$
\limsup\limits_{t\rightarrow\infty}\frac{h(t)}{t}\leq\lim\limits_{t\rightarrow\infty}\frac{k(t-T)}{t}=(1-\varepsilon)^{-2}s,
$$
which follows that
$$
\limsup\limits_{t\rightarrow\infty}\frac{h(t)}{t}\leq2\max\left\{\sqrt{\lambda},\sqrt{d(\nu+c)}\right\}
$$
by the arbitrariness of $\varepsilon$ and $s>2\max\{\sqrt{\lambda},\sqrt{d(\nu+c)}\}$.
\end{proof}

\begin{remark}
Theorem \ref{th5.1} shows that when spreading occurs, the asymptotic spreading speed of $h(t)$ cannot be faster than  $2\max\{\sqrt{\lambda},\sqrt{d(\nu+c)}\}$.
\end{remark}

\begin{theorem}\label{th5.2}
Assume that $s_i(\infty)$, $k_i(\infty)=\infty$ $(i=1,2)$. Let $(\phi_i,s_i)$, $(\psi_i,k_i)$ be solutions of the free boundary problems
$$
\begin{cases}
\phi_{1t}-\phi_{1xx}=\lambda \phi_1-\phi_1^2,~&t>0,0<x<s_1(t),\\
\phi_{1x}(t,0)=\phi_1(t,s_1(t))=0,   &t>0,\\
s_1^{\prime}(t)=-\kappa_1\phi_1(t,s_2(t)),&t>0,\\
\phi_1(0,x)=\phi_{10},&x\in[0,s_{10}],\\
s_1(0)=s_{10},
\end{cases}
$$
$$
\begin{cases}
\phi_{2t}-\phi_{2xx}=(\lambda-\frac{b}{m}) \phi_2-\phi_2^2,~&t>0,0<x<s_2(t),\\
\phi_{2x}(t,0)=\phi_2(t,s_2(t))=0,   &t>0,\\
s_2^{\prime}(t)=-\kappa_2\phi_2(t,s_2(t)),&t>0,\\
\phi_2(0,x)=\phi_{20},&x\in[0,s_{20}],\\
s_2(0)=s_{20},
\end{cases}
$$

$$
\begin{cases}
\psi_{1t}-\psi_{1xx}=(\nu+c)\psi_1-\psi_1^2,~&t>0,0<x<k_1(t),\\
\psi_{1x}(t,0)=\psi_1(t,k_1(t))=0,   &t>0,\\
k_1^{\prime}(t)=-\tau_1\psi_1(t,k_1(t)),&t>0,\\
\psi_1(0,x)=\psi_{10},&x\in[0,k_{10}],\\
k_1(0)=k_{10},
\end{cases}
$$

$$
\begin{cases}
\psi_{2t}-\psi_{2xx}=\nu\psi_2-\psi_2^2,~&t>0,0<x<k_2(t),\\
\psi_{2x}(t,0)=\psi_2(t,k_2(t))=0,   &t>0,\\
k_2^{\prime}(t)=-\tau_2\psi_2(t,k_2(t)),&t>0,\\
\psi_2(0,x)=\psi_{20},&x\in[0,k_{20}],\\
k_2(0)=k_{20},
\end{cases}
$$
respectively, where $k_i$, $\phi_{i0}$, $s_{i0}$, $\tau_i$ are positive constants. By Theorem 4.2 of \cite{Du1}, there exist positive constants $s^*$, $s_*$, $k^*$, $k_*$  respectively such that
$$
\lim\limits_{t\rightarrow\infty}\frac{s_1(t)}{t}=s^*,~\lim\limits_{t\rightarrow\infty}\frac{s_2(t)}{t}=s_*,~\lim\limits_{t\rightarrow\infty}\frac{k_1(t)}{t}=k^*,~\lim\limits_{t\rightarrow\infty}\frac{k_2(t)}{t}=k_*.
$$
Suppose that $\kappa_1\geq\mu$, $\kappa_2\leq\mu$, $\tau_1\geq\mu\rho$, $\tau_2\leq\mu\rho$ and
$$
\phi_{10}\geq u_0,~s_{10}\geq h_0,~\phi_{20}\leq u_0,~s_{20}\leq h_0,
$$
$$
\psi_{10}\geq v_0,~k_{10}\geq h_0,~\psi_{20}\leq v_0,~k_{20}\leq h_0.
$$
As a result of Lemma \ref{lm2.2}, \ref{lm2.3} and Remark \ref{rm3}, we have $s_1(t),k_1(t)\leq h(t)\leq s_2(t),k_2(t)$. Therefore,
$$
\max\{s_*,k_*\}\leq\liminf\limits_{t\rightarrow\infty}\frac{h(t)}{t},~\limsup\limits_{t\rightarrow\infty}\frac{h(t)}{t}\leq\min\{s^*,k^*\}.
$$
\end{theorem}

\section{Discussion}
In this paper, we have studied a ratio-dependent predator-prey model with a Neumann boundary on left side representing that the left boundary is fixed, and a free boundary $x=h(t)$ concerned with both prey and predator on right side, which describes the movement process for both prey and predator species. Firstly, a spreading-vanishing dichotomy and the criteria for spreading and vanishing are given, as summarised below.

(i)(Spreading case)If the size of initial habitat of prey and predator is equal to or more than $\Lambda:=\frac{\pi}{2}\min\{\sqrt{\frac{m}{m\lambda-b}},\sqrt{\frac{d}{\nu+c}}\}$, or less than $\Lambda$ but the moving coefficient $\mu$ of the free boundary is over than some positive constant $\overline \mu$ which depends on $u_0$, $v_0$ and $h_0$, then both species will spread successfully. In addition, as $t$ goes to infinity, the prey and predator goes to their stationary solutions $u^*$ and $v^*$, respectively.

(ii)(Vanishing case)While if the size of initial habitat is less than $\Lambda$ and the moving coefficient $\mu$ of the free boundary $h(t)$ is not more than $\underline \mu$ which also depends on $u_0$, $v_0$ and $h_0$, then the two species will vanish eventually. Moreover, as $t\rightarrow\infty$, the free boundary is limited to $\Lambda$.

When spreading occurs, we then estimate the asymptotic spread of the free boundary $x=h(t)$. We provide an upper bound for $\limsup\limits_{t\rightarrow\infty}\frac{h(t)}{t}$ which is $2\max\{\sqrt{\lambda},\sqrt{d(\nu+c)}\}$ (Theorem \ref{th5.1}), and give the scope of $\frac{h(t)}{t}$, which is not less than $\max\{s_*,k_*\}$ and not more than $\min\{s^*,k^*\}$ (Theorem \ref{th5.2}).

The positive constant ``$\Lambda$" is a vital threshold to judge whether it is spreading or not (More explanations see \cite{Du1}). In order to get a more accurate number, we study the waves of finite length to construct a lower solution of (\ref{Q}). Thus, we get a smaller number $\frac{\pi}{2}\sqrt{\frac{d}{\nu+c}}$ than the previous number $\frac{\pi}{2}\sqrt{\frac{d}{\nu}}$.

When vanishing happens, in this paper both prey and predator will die out eventually, while in \cite{liu2020free} only prey will be vanished, which is the most important difference between $h(t)$ depending on both prey and predator and that depending only on prey. In the natural world, predator that only lives on this prey will not be able to survive if prey goes extinct. Intuitively, the result in this paper seems to be closer to reality.

The above conclusions are instructive for us. Assume that predator $v$ only live on this prey $u$. Then two species co-exist, that is, when a new or an invasive species invade, either two species die out eventually or if the local species can escape to the whole space then the invasive species will widespread the whole space. In order to protect the local species, we can (i)enlarge the initial habitat of local species, (ii)increase the coefficient of the free boundary. Moreover, introducing natural enemy and taking the opposite approaches above are effective methods to control pest species.

\section*{Acknowledgments}
The author's work was supported by NSFC (11671278). The author is also grateful for the anonymous reviewers for their helpful comments and suggestions.

\section*{Declarations of interest: none.}

\section*{References}
\bibliographystyle{plain}\setlength{\bibsep}{0ex}
\scriptsize
\bibliography{mybibfile}
\end{document}